\documentclass[]{article}

\usepackage{amsmath}
\usepackage{amssymb}
\usepackage{amsthm}
\usepackage[shortlabels]{enumitem}
\usepackage{mathtools}
\usepackage{cleveref}
\usepackage{resmes}
\usepackage{url}

\DeclareMathOperator*{\cl}{cl}

\DeclareMathOperator*{\DIV}{div}

\DeclareMathOperator*{\essinf}{ess-inf}

\DeclareFontFamily{U}{matha}{\hyphenchar\font45}
\DeclareFontShape{U}{matha}{m}{n}{
	<5> <6> <7> <8> <9> <10> gen * matha
	<10.95> matha10 <12> <14.4> <17.28> <20.74> <24.88> matha12
}{}
\DeclareSymbolFont{matha}{U}{matha}{m}{n}
\DeclareFontSubstitution{U}{matha}{m}{n}

\DeclareFontFamily{U}{mathx}{\hyphenchar\font45}
\DeclareFontShape{U}{mathx}{m}{n}{
	<5> <6> <7> <8> <9> <10>
	<10.95> <12> <14.4> <17.28> <20.74> <24.88>
	mathx10
}{}
\DeclareSymbolFont{mathx}{U}{mathx}{m}{n}
\DeclareFontSubstitution{U}{mathx}{m}{n}

\DeclareMathDelimiter{\vvvert}{0}{matha}{"7E}{mathx}{"17}

\def\e{\varepsilon}
\def\i{\infty}
\def\p{\partial}

\def\N{{\mathbb N}}

\def\R{{\mathbb R}}
\def\AA{{\mathcal A}}

\def\DD{{\mathcal D}}
\def\FF{{\mathcal F}}

\def\weak{\rightharpoonup}
\def\weakast{\overset{\ast}{\rightharpoonup}}

\newtheorem{conjecture}{Conjecture}[section]

\newtheorem{definition}{Definition}[section]
\newtheorem{lemma}{Lemma}[section]

\newtheorem{theorem}{Theorem}[section]

\def\XXint#1#2#3{{\setbox0=\hbox{$#1{#2#3}{\int}$} 
		\vcenter{\hbox{$#2#3$}}\kern-.5\wd0}}

\title{Some Preliminary Considerations on Energy Behavior in Fluid Dynamics}
\author{Thomas Ruf}

\begin{document}
	
	\maketitle
	
	\begin{abstract}
		This work presents a tentative discussion of certain aspects of energy behavior in the context of mathematical fluid dynamics. While some observations are made regarding certain patterns in energy behavior under particular conditions, the broader implications of these findings remain uncertain and should be interpreted with considerable caution. The results are preliminary in nature, and their relevance to analytic properties of solutions is demanding clarification at this stage. These considerations are intended to motivate further inquiry rather than to establish any definitive conclusions. Readers should approach the material presented here as exploratory, with significant open questions left unresolved.
	\end{abstract}
	
	\section{Introduction}
	
	The study of energy behavior in mathematical fluid dynamics has long been a topic of interest, offering a rich landscape for exploration and inquiry. Despite significant progress in understanding various aspects of fluid motion, many questions remain unanswered, particularly when it comes to the intricate interplay between energy dynamics and the analytic properties of solutions. This work seeks to contribute to this ongoing dialogue by presenting a tentative discussion of certain patterns in energy behavior under specific conditions. While the observations made here are preliminary, they are intended to provide a foundation for further investigation into this complex and multifaceted area.
	\\
	The motivation for this work stems from the broader context of mathematical fluid dynamics, where the Navier-Stokes equations serve as a cornerstone for modeling incompressible fluid flow. These equations, while deceptively simple in form, encapsulate a wealth of mathematical challenges that continue to inspire research across disciplines. Among these challenges, the behavior of energy in weak solutions remains a particularly intriguing and elusive subject. By focusing on certain aspects of energy behavior, this work aims to shed light on patterns that may hold relevance for understanding the broader analytic properties of solutions. However, it is important to emphasize that the findings presented here are exploratory in nature and should not be interpreted as definitive.
	\\
	The results discussed in this paper are framed within a specific set of assumptions, which, while carefully chosen, may limit the generality of the conclusions. As such, the broader implications of these findings remain uncertain and demand further clarification. This work aims to provide a starting point for future inquiry. Readers are encouraged to approach the material with an open mind, recognizing that significant open questions remain and that the observations presented here are but one step in a much larger journey.
	\\
	In presenting this work, we hope to contribute to the ongoing conversation surrounding energy behavior in fluid dynamics, while acknowledging the inherent limitations of our approach. The tentative nature of the results reflects the complexity of the subject and the need for continued exploration. It is our hope that this discussion will inspire further research and provide a basis for deeper engagement with the many unresolved questions that characterize this field.
	
	\section{Preliminaries}
	
	Let $T \in \left( 0, \i \right]$ be a finite or infinite time horizon so that
	\begin{equation*}
		\cl \left[ 0, T \right) =
		\begin{cases}
			\left[ 0, T \right] & \text{if } T < \i, \\
			\left[ 0, T \right) & \text{if } T = \i.
		\end{cases}
	\end{equation*}
	We briefly recall in this section some notation and important background material from the theory of the Navier-Stokes system
	\begin{equation} \label{eq:NS}
		\begin{cases}
			\p_t u + (u \cdot \nabla) u - \Delta u = - \nabla p & \text{in } (0, T) \times \R^3, \\
			\DIV u = 0 & \text{in } (0, T) \times \R^3, \\
			u(0, x) = u_0(x) & \text{on } \R^3.
		\end{cases}
	\end{equation}
	Here $u \colon (0, T) \times \R^3 \to \R^3$ is the velocity field, $p \colon (0, T) \times \R^3 \to \R$ the pressure, and $u_0 \in L^2_\sigma(\R^3)$ is an initial datum. The space $L^2_\sigma(\R^3)$ is the closure of divergence-free test functions in the $L^2$-norm. Similarly, $H^1_\sigma(\R^3)$ is the intersection of the Sobolev space $H^1(\R^3; \R^3)$ with $L^2_\sigma(\R^3)$. When considering vector-valued function spaces such as $L^p(\R^3; \R^3)$, we often abbreviate to forms such as $L^p(\R^3)$ if the range space is clear from context. Notations such as $\DIV$, $\nabla$, or $\Delta$ frequently refer to the divergence, gradient, or Laplacian with respect to the spatial variables only. Similarly, notations such as $C^\i_{c, \sigma}( (0, T) \times \R^3)$ denote spaces of $\R^3$-valued functions whose spatial divergence vanishes. A useful identity to keep in mind is
	\begin{equation} \label{eq:convec id}
		(u \cdot \nabla) u = \DIV \left( u \otimes u \right) \text{ if } \DIV u = 0
	\end{equation}
	where $\otimes$ denotes the tensor product. Three notions of solutions for \eqref{eq:NS} will be needed in the following. These are \emph{weak solutions} as defined in \cite[Ch. 3]{RRS}, \emph{Leray-Hopf weak solutions} as in \cite[Ch. 4.3]{RRS}, and \emph{strong solutions} as in \cite[Ch. 6]{RRS}. We extend these definitions to the case $T = \i$ by calling a solution weak etc. if it is weak etc. on every compact subinterval of its time domain. For frequent latter use, we record here a general property of the  convective term $(v \cdot \nabla ) u$.
	
	\begin{lemma} \label{lem:convec van}
		Let $u \in C^1_c(\R^d; \R^d)$, $v \in C^1_\sigma(\R^d; \R^d)$, and $f \in C( \R ; \R)$. Then
		\begin{equation} \label{eq:convec van}
			\int_{\R^d} f \left( \left| u \right| \right) \langle \left( v \cdot \nabla \right) u, u \rangle \, \mathrm{d} x = 0.
		\end{equation}
	\end{lemma}
	
	\begin{proof}
		Let
		$$
		F(t) = \int_0^{t^2} f \left( \sqrt{s} \right) \, \mathrm{d} s, \quad t \in \R.
		$$
		Using the product rule and $\DIV v = 0$, we calculate
		$$
		\DIV \left( F \left( \left| u \right| \right) v \right)
		= \nabla \left( F \circ \left| u \right| \right) \cdot v
		= 2 \sum_{i, j = 1}^d v_i f \left( \left| u \right| \right) \p_i u_j u_j
		= 2 f \left( \left| u \right| \right) \langle \left( v \cdot \nabla \right) u, u \rangle.
		$$
		Integrating this identity over $\R^d$, using the compact support of $u$ and the divergence theorem, we obtain \eqref{eq:convec van}.
	\end{proof}
	
	\section{Approximation} \label{sec:approx}
	
	In this section, we define the approximate system to the Navier–Stokes equations, for which we prove existence of a unique weak solution and energy estimates by recourse to standard theory for abstract evolution equations. To achieve this, we regularize the convective term $\left( u \cdot \nabla \right) u$ by truncation and add a tightening linear zero order term to the equations, i.e., one that penalizes growth at infinity in the space variables so that the natural energy space of the approximate system enjoys a compact embedding into $L^2(\R^3)$. Our regularization for the convective term is a modification of \cite[§3]{W}. The tightening term will then be sent to zero and the solutions shown to converge accordingly, leading to a well-posed approximate system that differs from \eqref{eq:NS} only by a regularization of the convective term.
	
	\subsection{Setup}
	
	Let $u_0 \in H^1_\sigma(\R^3)$ be a given initial value. Since a compact subset of $L^\alpha$ for $1 \le \alpha < \i$ is necessarily tight by the Vitali convergence theorem, \Cref{lem:tightness} provides a measurable function $w \colon \R^3 \to \left[ 0, \i \right]$ such that $\sqrt{w}$ tightens $u_0$ in $L^2(\R^3)$ in the sense of \Cref{sec:appendix}. Explicitly, this means that
	$$
	\int_{\R^3} w(x) \left| u_0(x) \right|^2 \, \mathrm{d} x < \i
	$$
	while there are measurable sets $A_n \subset \R^3$ of finite measure such that
	$$
	\lim_{n \to \i} \essinf_{x \in \R^3 \setminus A_n} w(x) = \i.
	$$
	We may assume that
	\begin{equation} \label{eq:w prop}
		w \text{ is positively bounded below and locally bounded above }
	\end{equation}
	by possibly modifying it via
	$$
	w_{\text{new} }(x) = \max\{ 1, \min\{ \left| x \right|, w(x) \} \}.
	$$
	Let $\rho \in C^1(\R)$ be a function satisfying $\rho'(0) = 0$ and
	\begin{equation} \label{eq:rho grwth}
		\sup_{x \in \R} \left| x \rho(x) \right| + \sup_{x \in \R} \left| x \rho'(x) \right| < \i.
	\end{equation}
	We define $\mathfrak{r} \colon \R^3 \to \R^3$ by $\mathfrak{r}(x) = \rho \left( \left| x \right| \right) x$ where $\mathfrak{r} \in C^1(\R^3)$ since $\rho'(0) = 0$.
	\begin{equation} \label{eq:r grwth}
		\left| \mathfrak{r}(u) \right| \le C \min \{ 1, \left| u \right| \} \quad \forall u \in \R^3,
	\end{equation}
	\begin{equation} \label{eq:Dr grwth}
		\left| D \mathfrak{r}(u) \right| \le C \quad \forall u \in \R^3.
	\end{equation}
	Accordingly, setting $\mathfrak{R}(u) = \mathfrak{r}(u) \otimes \mathfrak{r}(u)$, we have
	\begin{equation} \label{eq:R grwth}
		\left| \mathfrak{R}(u) \right| \le C \min \{ 1, \left| u \right|^2 \} \quad \forall u \in \R^3,
	\end{equation}
	\begin{equation} \label{eq:DR grwth}
		\left| D \mathfrak{R}(u) \right| \le C \min \{ 1, \left| u \right| \} \quad \forall u \in \R^3.
	\end{equation}
	In particular, $\mathfrak{R}$ induces superposition operators
	\begin{equation} \label{eq:R superpos Lebesgue}
		\begin{gathered}
			u \mapsto \mathfrak{R} \circ u \colon L^2(\R^3; \R^3) \to \left( L^1 \cap L^\i \right)(\R^3; \R^{3 \times 3}_{\mathrm{sym}} ), \\
			\left| \mathfrak{R} \circ u \right|_{L^1} \le C \| \rho \|_\i \left| u \right|^2_{L^2}, \quad
			\left| \mathfrak{R} \circ u \right|_{L^\i} \le \| \mathfrak{R} \|_\i
		\end{gathered}
	\end{equation}
	and
	\begin{equation} \label{eq:R superpos Sobolev}
		\begin{gathered}
			u \mapsto \mathfrak{R} \circ u \colon H^1(\R^3 ; \R^3) \to H^1(\R^3; \R^{3 \times 3}_{\mathrm{sym}} ), \\
			\left| D \left( \mathfrak{R} \circ u \right) \right|_{L^2} \le \| D \mathfrak{R} \|_\i \left| \nabla u \right|_{L^2}
		\end{gathered}
	\end{equation}
	by the Sobolev chain rule \cite[Thm. 4.4]{EG}. For $\e > 0$, we consider the modified system
	\begin{equation} \label{eq:NS mod}
		\begin{cases}
			\p_t v_\e + \DIV \left( \mathfrak{R} \circ v_\e \right) - \Delta v_\e + \e w v_\e = - \nabla p_\e & \text{in } (0, T) \times \R^3, \\
			\DIV v_\e = 0 & \text{in } (0, T) \times \R^3, \\
			v_\e(0, x) = u_0(x) & \text{on } \R^3.
		\end{cases}
	\end{equation}
	To analyze \eqref{eq:NS mod}, we introduce the following functional framework:
	\\
	
	\emph{Function spaces.} Let
	$$
	H = L^2_\sigma(\R^3),
	\quad V = V_\e = H^1_\sigma(\R^3) \cap L^2(\e w(x) \, \mathrm{d} x).
	$$
	We remark that the space $V$ is dependent on the weight $\e w$ and thereby on the initial condition $u_0$ if $\e > 0$. Clearly $V_0 = H^1_\sigma(\R^3)$, but $V_{\e_1} = V_{\e_2}$ for all $\e_1, \e_2 > 0$. To verify that $\left( V, H, V^* \right)$ forms a Gelfand triple, as defined, e.g., in \cite[Ch. 7.2]{R}, note that $V$ is separable and reflexive by standard arguments, with a dense continuous embedding $V \hookrightarrow H$ given by identical inclusion. It is dense since even divergence-free test functions are dense in $H$ by definition.
	Moreover, the embedding $V_\e \hookrightarrow H$ is compact if $\e > 0$: by \Cref{lem:tightness}, every bounded subset of $L^2(w(x) \, \mathrm{d} x)$ is tight in $L^2(\R^3)$. Combining this with the Rellich-Kondrachov theorem gives compactness.
	\\
	
	\emph{Operators}. Define the operators $A_1, A_2 \colon V \to V^*$ by
	$$
	A_1(u) = -\Delta u + \e w u,
	\quad
	A_2(u) = \DIV \left( \mathfrak{R} \circ u \right),
	\quad
	A = A_1 + A_2.
	$$
	The operator $A_1$ is the Fréchet derivative of the convex quadratic functional
	$$
	\phi_\e(u) = \frac{1}{2} \int_{\R^3} \left| \nabla u \right|^2 + \e w \left| u \right|^2 \, \mathrm{d} x.
	$$
	The nonlinear operator $A_2$ maps $V$ into $L^2(\R^3; \R^3)$ by \eqref{eq:R superpos Sobolev}. Considering
	$$
	A_2 \colon V \to H^*, \quad \langle A_2(u), v \rangle = \int_{\R^3} \langle \DIV \left( \mathfrak{R} \circ u \right), v \rangle \, \mathrm{d} x
	$$
	and identifying $H \cong H^*$, one can view $A_2$ as an $H$-valued mapping.
	
	\begin{definition} \label{def:sol}
		Let $\e \ge 0$ and $u_0 \in H^1_\sigma(\R^3)$. Let $w \colon \R^3 \to \left[ 0, \i \right]$ be a weight tightening $u_0$ in $L^2(\R^3)$ and satisfying \eqref{eq:w prop}. A \emph{solution} to the Cauchy problem \eqref{eq:NS mod} is a function
		\begin{equation} \label{eq:veps space}
			v_\e \in W^{1, 1}_{\mathrm{loc} } \left( 0, T ; V^* \right) \cap L^1_{\mathrm{loc} } \left( 0, T ; V \right)
		\end{equation}
		satisfying the weak formulation
		\begin{equation} \label{eq:v solves}
			\begin{gathered}
				\int_0^T \int_{\R^3} v'_\e(t, x) \cdot \varphi(t, x) \, \mathrm{d} x \, \mathrm{d} t = \\
				- \int_0^T \langle A_1(v_\e(t) ) + A_2(v_\e(t) ), \varphi(t, \cdot) \rangle_{V^*, V} \, \mathrm{d} t
				\quad \forall \varphi \in C^\i_{c, \sigma} \left( (0, T) \times \R^3 \right)
			\end{gathered}
		\end{equation}
		and the initial condition
		\begin{equation} \label{eq:v initiates}
			\lim_{t \to 0^+} \int_{\R^3} v_\e(t, x) \cdot \psi(x) \, \mathrm{d} x = \int_{\R^3} u_0(x) \cdot \psi(x) \, \mathrm{d} x \quad \forall \psi \in C^\i_{c, \sigma}(\R^3).
		\end{equation}
	\end{definition}
	
	Note that \eqref{eq:veps space} implies $A(v_\e) \in L^1_{\mathrm{loc} }(0, T; V^*)$ since $A$ grows at most linearly, being a sum of the continuous linear operator $A_1$ and the non-linear mapping $A_2$, whose growth is linearly controlled in $V^*$ since \eqref{eq:R superpos Lebesgue} implies
	$$
	\left| \mathfrak{R} \circ u \right|_{L^2}
	\le \left| \mathfrak{R} \circ u \right|_{L^1}^{1/2} \left| \mathfrak{R} \circ u \right|_{L^\i}^{1/2}
	\le C \| \rho \|_\i^{1/2} \| \mathfrak{R} \|_\i^{1/2} \left| u \right|_{L^2}.
	$$
	Moreover, $C^\i_{c, \sigma} \left( (0, T) \times \R^3 \right)$ is weak* dense in $L^\i(0, T; V)$.
	Consequently, if \eqref{eq:veps space}, then \eqref{eq:v solves} is equivalent to
	\begin{equation} \label{eq:v solves2}
		v'_\e(t) = - A(v_\e(t) ) \text{ in } V^* \text{ for a.e. } t \in (0, T).
	\end{equation}
	
	\subsection{Existence and some energy identities}
	
	We are now ready to prove existence of a solution to \eqref{eq:NS mod} that enjoys certain energy estimates and is uniquely determined as the only solution satisfying the regularity properties following from these estimates.
	
	\begin{lemma} \label{lem:weak appr sol}
		Let $\e > 0$ and $u_0 \in H^1_\sigma(\R^3)$. For every weight function $w \colon \R^3 \to \left[ 0, \i \right)$ that tightens $u_0$ in $L^2(\R^3)$ and satisfies \eqref{eq:w prop}, there exists a solution $v_\e$ to the Cauchy problem \eqref{eq:NS mod} satisfying
		\begin{equation} \label{eq:weak slt reg}
			v_\e \in H^1_{\mathrm{loc} } \left( \cl \left[ 0, T \right) ; H \right) \cap L^\i_{\mathrm{loc} } \left( \cl \left[ 0, T \right) ; V \right).
		\end{equation}
	\end{lemma}
	
	\begin{proof}
		We check the assumptions to apply \cite[Thm. 8.16]{R}, using the notation and terminology of the reference. While it is not explicitly stated in \cite[Thm. 8.16]{R}, the solution constructed there is strong in the sense of \cite[Def. 8.1]{R}, making it compatible with the notion of solution given by \Cref{def:sol} and \eqref{eq:v solves2}. A standard diagonal argument shows that it suffices to obtain the convergence conclusion of \cite[Thm. 8.16(ii)]{R} for an arbitrarily long but compact time interval $I \subset \cl \left[ 0, T \right)$ with $0 \in I$. Thus, without restricting generality, we may assume $T < \i$ so that $\cl \left[ 0, T \right) = \left[ 0, T \right]$.
		\\
		We have $u_0 \in V$ since the weight $w$ tightens $u_0$ in $L^2(\R^3)$. No inhomogeneity $f$ is present. The functional $\phi_\e$ is convex. It also being continuous and quadratic, the operator $D \phi_\e = A_1$ is continuous and linear. In particular, $A_1$ is bounded and radially continuous. The estimate
		\begin{equation} \label{eq:phi eps lwr grwth}
			\phi_\e(u) \ge \frac{\e}{2} \left| u \right|_V^2
		\end{equation}
		gives the required lower growth of $\phi_\e$.
		\\
		The superposition operator $V \to L^2(\R^3) \colon u \mapsto \mathfrak{R} \circ u$ underlying $A_2$ is well-defined and continuous by \eqref{eq:R superpos Lebesgue}. Moreover, it even is compact by the same estimate and by compactness of the embedding $V \hookrightarrow H$. Clearly, this implies that $A_2$ is totally continuous.	It remains to check the growth control on $\left| A_2(u) \right|_H$. The desired estimate is
		\begin{equation} \label{eq:A2 uppr grwth}
			\begin{gathered}
				\left| A_2(u) \right|_H
				= \left| A_2(u) \right|_{H^*}
				\le \left| A_2(u) \right|_{L^2(\R^3) }
				\le \left| D \mathfrak{R} \circ u \right|_{L^\i(\R^3) } \left| \nabla u \right|_{L^2(\R^3) } \\
				\le \| D \mathfrak{R} \|_\i \left| \nabla u \right|_{L^2(\R^3) }
				\le C \left| u \right|_V
			\end{gathered}
		\end{equation}
		by \eqref{eq:DR grwth} or \eqref{eq:R superpos Sobolev}. Finally, the operator $A = A_1 + A_2$ is pseudomonotone as a sum of the monotone operator $A_1$ and the compact operator $A_2$; It is semicoercive because
		\begin{equation*}
			\langle A(u), u \rangle = \langle A_1(u), u \rangle + \langle A_2(u), u \rangle = 2 \phi_\e(u) + \langle A_2(u), u \rangle \ge \e \left| u \right|_V^2 - C \left| u \right|_V
		\end{equation*}
		by \eqref{eq:phi eps lwr grwth} and \eqref{eq:A2 uppr grwth}. In total, \cite[Thm. 8.16]{R} implies the claim.
	\end{proof}
	
	\begin{lemma} \label{lem:appr sol uni}
		For every $\e \ge 0$ and $u_0 \in H^1_\sigma(\R^3)$ with a weight $w \colon \R^3 \to \left[0, \i \right)$ that tightens $u_0$ in $L^2(\R^3)$ and satisfies \eqref{eq:w prop}, there is at most one solution to the Cauchy problem \eqref{eq:NS mod} satisfying \eqref{eq:weak slt reg}.
	\end{lemma}
	
	\begin{proof}
		This is a standard argument using the Gronwall lemma. We carry it out for the sake of completeness. Let $u, v$ be such solutions and set $d = u - v$. Since $u$ and $v$ both solve \eqref{eq:NS mod}, we can test the equation $\p_t d + A_1(d) = A_2(v) - A_2(u)$ with $d$ and use the theory of Gelfand triples, cf., e.g., \cite[Ch. 7.2, Rem. 7.5]{R}, to find
		\begin{align*}
			\frac{1}{2} \frac{d}{dt} \left| d(t) \right|^2_H + \left| \nabla d(t) \right|^2_H
			& \le \frac{1}{2} \frac{d}{dt} \left| d(t) \right|^2_H + \left| \nabla d(t) \right|^2_H + \e \left| d(t) \right|^2_{L^2(w \, \mathrm{d} x) } \\
			& = \langle A_2(v) - A_2(u), d \rangle_H(t) \\
			& = \langle \mathfrak{R} \circ u - \mathfrak{R} \circ v, \nabla d \rangle_H(t) \\
			& \le \left| \mathfrak{R} \circ u - \mathfrak{R} \circ v \right|_{L^2}(t) \left| \nabla d(t) \right|_H \\
			& \le \frac{1}{4} \left| \mathfrak{R} \circ u - \mathfrak{R} \circ v \right|^2_{L^2}(t) + \left| \nabla d(t) \right|^2_H \\
			& \le \frac{\mathrm{Lip}(\mathfrak{R} )^2}{4} \left| d(t) \right|^2_H + \left| \nabla d(t) \right|^2_H
		\end{align*}
		by the Young inequality. By \eqref{eq:DR grwth}, we have
		$$
		\mathrm{Lip}(\mathfrak{R} ) \le \| D \mathfrak{R} \|_\i < \i.
		$$
		Consequently, rearranging terms, we have shown
		$$
		\frac{d}{dt} \left| d(t) \right|^2_H \le C \left| d(t) \right|^2_H \quad \text{ for a.e. } t \in (0, T)
		$$
		so that $d(0) = 0 \implies d(t) \equiv 0$ and hence uniqueness.
	\end{proof}
	
	\begin{lemma} \label{lem:appr sol en}
		Let $\e \ge 0$ and $u_0 \in H^1_\sigma(\R^3)$ with a weight $w \colon \R^3 \to \left[0, \i \right)$ that tightens $u_0$ in $L^2(\R^3)$ and satisfies \eqref{eq:w prop}. Every solution $v_\e$ to the Cauchy problem \eqref{eq:NS mod} fulfilling \eqref{eq:weak slt reg} satisfies the energy identities
		\begin{equation} \label{eq:Ham appr en id}
			\begin{gathered}
				\frac{1}{2} \int_{\R^3} \left| v_\e(t) \right|^2 \, \mathrm{d} x + \int_0^t \int_{\R^3} \left| \nabla v_\e \right|^2 + \e w \left| v_\e \right|^2 \, \mathrm{d} x \, \mathrm{d} s \\
				= \frac{1}{2} \int_{\R^3} \left| u_0 \right|^2 \, \mathrm{d} x \quad \forall t \in \cl \left[ 0, T \right)
			\end{gathered}
		\end{equation}
		and
		\begin{equation} \label{eq:Gflow appr en id}
			\begin{gathered}
				\int_0^t \int_{\R^3} \left| v_\e' \right|^2 + \langle \DIV \left( \mathfrak{R} \circ v_\e \right), v_\e' \rangle \, \mathrm{d} x \, \mathrm{d} s
				+ \frac{1}{2} \int_{\R^3} \left| \nabla v_\e(t) \right|^2 + \e w \left| v_\e(t) \right|^2 \, \mathrm{d} x
				\\
				= \frac{1}{2} \int_{\R^3} \left| \nabla u_0 \right|^2 + \e w \left| u_0 \right|^2 \, \mathrm{d} x \quad \forall t \in \cl \left[ 0, T \right).
			\end{gathered}
		\end{equation}
	\end{lemma}
	
	\begin{proof}
		Any solution $v_\e$ satisfies \eqref{eq:v solves2}, which together with \eqref{eq:weak slt reg} implies
		\begin{equation} \label{eq:op eq in H}
			v'_\e(t) = - A(v_\e(t) ) \text{ a.e. in } H
		\end{equation}
		under the identification of $H$ with a subspace of $V^*$ as usual in the theory of Gelfand triples \cite[Ch. 7.2]{R}. We test \eqref{eq:op eq in H} with $v_\e(t)$ to get
		\begin{equation} \label{eq:1st step}
			\langle v'_\e(t), v_\e(t) \rangle_H = - \langle A(v_\e(t) ), v_\e(t) \rangle_H = - \langle A(v_\e(t) ), v_\e(t) \rangle_{V^*, V}
		\end{equation}
		for a.e. $t \in (0, T)$. Now
		\begin{equation} \label{eq:2nd step}
			\langle A_1(v_\e(t) ), v_\e(t) \rangle = 2 \phi_\e(v_\e(t) ),
		\end{equation}
		and
		\begin{equation} \label{eq:3rd step}
			\begin{aligned}
				\langle A_2(v_\e(t) ), v_\e(t) \rangle
				& = \int_{\R^3} \langle \DIV \left( \mathfrak{R} \circ v_\e(t) \right), v_\e(t) \rangle \, \mathrm{d} x \\
				& = - \int_{\R^3} \langle \mathfrak{R} \circ v_\e(t), \nabla v_\e(t) \rangle \, \mathrm{d} x \\
				& = - \sum_{i, j = 1}^3 \int_{\R^3} \langle \mathfrak{r}(v_\e)_i \mathfrak{r}(v_\e)_j, \p_i v_{\e, j} \rangle \, \mathrm{d} x \\
				& = - \sum_{i, j = 1}^3 \int_{\R^3} \rho^2 \left( \left| v_\e \right| \right) \langle v_{\e, i} v_{\e, j} , \p_i v_{\e, j} \rangle \, \mathrm{d} x \\
				& = - \int_{\R^3} \rho^2 \left( \left| v_\e \right| \right) \langle \left( v_\e \nabla \right) v_\e, v_\e \rangle \, \mathrm{d} x \\
				& = 0
			\end{aligned}
		\end{equation}
		for a.e. $t \in (0, T)$ by $\DIV v_\e = 0$ and \Cref{lem:convec van}, using also the continuity of $A_2$ by \eqref{eq:R superpos Sobolev} together with the denseness of $C^\i_{c, \sigma}(\R^3)$ in $H^1_\sigma(\R^3)$. By \eqref{eq:weak slt reg}, we have
		\begin{equation} \label{eq:LionsMag}
			\langle v'_\e(t), v_\e(t) \rangle_H = \frac{1}{2} \frac{d}{dt} \left| v_\e(t) \right|^2_H.
		\end{equation}
		Inserting \eqref{eq:2nd step}, \eqref{eq:3rd step}, \eqref{eq:LionsMag} into \eqref{eq:1st step} and integrating the result over $(0, t)$, we conclude \eqref{eq:Ham appr en id} in the form
		$$
		\frac{1}{2} \left| v_\e(t) \right|^2_H + 2 \int_0^t \phi_\e(v_\e(s) ) \, \mathrm{d} s = \frac{1}{2} \left| u_0 \right|^2_H \quad \forall t \ge 0.
		$$
		Regarding \eqref{eq:Gflow appr en id}, note that
		$$
		D \mathfrak{R} \circ v_\e \in L^\i \left( \cl \left[ 0, T \right) \times \R^3 \right)
		$$
		by \eqref{eq:DR grwth}. As a consequence
		\begin{equation} \label{eq:A2 bound}
			A_2(v_\e) = \DIV \left( \mathfrak{R} \circ u \right) \in L^2_{\mathrm{loc} } \left( \cl \left[ 0, T \right) ; L^2 \left( \R^3 \right) \right)
		\end{equation}
		by the Sobolev chain rule so that \eqref{eq:weak slt reg} and \eqref{eq:op eq in H} imply
		\begin{equation} \label{eq:A1 bound}
			A_1(v_\e) \in L^2_{\mathrm{loc} } \left( \cl \left[ 0, T \right) ; L^2 \left( \R^3 \right) \right).
		\end{equation}
		Multiplying \eqref{eq:op eq in H} with $v'_\e$, integrating the result over $(0, t)$, and using \eqref{eq:A2 bound}, \eqref{eq:A1 bound} together with $A_1 = D \phi_\e$ and the chain rule \cite[Lem. 9.1]{R}, we conclude \eqref{eq:Gflow appr en id}.
	\end{proof}
	
	Keeping intact the notation introduced at the beginning of \Cref{sec:approx}, we are now prepared to send $\e \to 0$, passing from \eqref{eq:NS mod} to the more accurate approximate system
	\begin{equation} \label{eq:NS mod2}
		\begin{cases}
			\p_t v_\rho + \DIV \left( \mathfrak{R} \circ v_\rho \right) - \Delta v_\rho = - \nabla p_\rho & \text{in } (0, T) \times \R^3, \\
			\DIV v_\rho = 0 & \text{in } (0, T) \times \R^3, \\
			v_\rho(0, x) = u_0(x) & \text{on } \R^3.
		\end{cases}
	\end{equation}
	
	\begin{theorem} \label{thm:weak appr sol e=zero}
		For every $u_0 \in H^1_\sigma(\R^3)$, there exists a unique solution $v_\rho$ to the Cauchy problem \eqref{eq:NS mod2} satisfying
		\begin{equation} \label{eq:weak slt reg e=zero}
			v_\rho \in H^1_{\mathrm{loc} }( \cl \left[ 0, T \right) ; H ) \cap L^\i_{\mathrm{loc} } \left( \cl \left[ 0, T \right) ; H^1_\sigma(\R^3) \right).
		\end{equation}
		Moreover, the solution $v_\rho$ satisfies the energy identities
		\begin{equation} \label{eq:Ham appr en id2}
			\begin{gathered}
				\frac{1}{2} \int_{\R^3} \left| v_\rho(t) \right|^2 \, \mathrm{d} x + \int_0^t \int_{\R^3} \left| \nabla v_\rho \right|^2 \, \mathrm{d} x \, \mathrm{d} s \\
				= \frac{1}{2} \int_{\R^3} \left| u_0 \right|^2 \, \mathrm{d} x \quad \forall t \in \cl \left[ 0, T \right)
			\end{gathered}
		\end{equation}
		and
		\begin{equation} \label{eq:Gflow appr en id2}
			\begin{gathered}
				\int_0^t \int_{\R^3} \left| v'_\rho \right|^2 + \langle \DIV \left( \mathfrak{R} \circ v_\rho \right), v'_\rho \rangle \, \mathrm{d} x \, \mathrm{d} s
				+ \frac{1}{2} \int_{\R^3} \left| \nabla v_\rho(t) \right|^2 \, \mathrm{d} x
				\\
				= \frac{1}{2} \int_{\R^3} \left| \nabla u_0 \right|^2 \, \mathrm{d} x \quad \forall t \in \cl \left[ 0, T \right).
			\end{gathered}
		\end{equation}
	\end{theorem}
	
	\begin{proof}
		A solution to \eqref{eq:NS mod2} satisfying \eqref{eq:weak slt reg e=zero} is unique by \Cref{lem:appr sol uni} and satisfies the energy identities \eqref{eq:Ham appr en id2} and \eqref{eq:Gflow appr en id2} by \Cref{lem:appr sol en}. It remains to prove existence of such a solution. Let $\e > 0$ and fix $T^* \in \cl \left[ 0, T \right)$. For $\delta \in (0, 1)$, using the Young inequality and \eqref{eq:R superpos Sobolev} gives
		\begin{gather*}
			\int_0^t \int_{\R^3} \langle \DIV \left( \mathfrak{R} \circ v_\e \right), v_\e' \rangle \, \mathrm{d} x \, \mathrm{d} s \\
			\ge - C_\delta \| D \mathfrak{R} \|_\i \int_0^t \int_{\R^3} \left| \nabla v_\e \right|^2 \, \mathrm{d} x \, \mathrm{d} s - \delta \int_0^t \int_{\R^3} \left| v_\e' \right|^2 \, \mathrm{d} x \, \mathrm{d} s.
		\end{gather*}
		Thus, adding a sufficiently large multiple of \eqref{eq:Ham appr en id} to \eqref{eq:Gflow appr en id} gives
		\begin{equation} \label{eq:veps bnd}
			\begin{gathered}
				\left| v_\e(t) \right|^2_H + \left| \nabla v_\e(t) \right|^2_H + \e \left| v_\e(t) \right|^2_{L^2(w \, \mathrm{d}x) } + \int_0^t \left| v'_\e(s) \right|^2_H \, \mathrm{d} s \\
				\le C = C \left( \delta, \right| u_0 \left|_{L^2}, \right| \nabla u_0 \left|_{L^2}, \| D \mathfrak{R} \|_\i \right) \quad \forall t \in \cl \left[ 0, T \right).
			\end{gathered}
		\end{equation}
		By \eqref{eq:veps bnd} and a diagonal argument, we may extract a subsequence $v_{\e_n}$ of $v_\e$ converging weakly* to some limit $v_\rho$ as $\e_n \to 0$ in the space
		$$
		H^1( 0, T^* ; H ) \cap L^\i \left( 0, T^* ; H^1_\sigma(\R^3) \right)
		$$
		for every $T^* \in \cl \left[ 0, T \right)$. In particular, by lower semicontinuity, we have \eqref{eq:veps bnd} with $v_\rho$ instead of $v_\e$ so that \eqref{eq:weak slt reg e=zero} follows.	Now, recalling \Cref{def:sol} and passing to the limit in \eqref{eq:NS mod}, note that the term
		$$
		\int_0^T \int_{\R^3} \e w(x) v_\e(t, x) \cdot \varphi(t, x) \, \mathrm{d} x \, \mathrm{d} t, \quad \varphi \in C^\i_{c, \sigma} \left( (0, T) \times \R^3 \right)
		$$
		goes to zero as $\e \to 0$ since $\sqrt{\e} v_\e$ is bounded in $L^\i(0, T^*; L^2(w \, \mathrm{d}x) )$. Also, all terms except $A_2(v_\e)$, are linear, hence their limits are immediate. The non-linearity $A_2$ is local and has sufficiently tame growth by \eqref{eq:R grwth} to treat it using the Rellich-Kondrachov compact embedding or the Aubin-Lions lemma. Consequently, $v_\rho$ solves \eqref{eq:NS mod2} in the sense of \Cref{def:sol}.
	\end{proof}
	
	\emph{Pressure}. So far, we have worked in pressure-free formulations of \eqref{eq:NS mod}, \eqref{eq:NS mod2}, investigating the systems on spaces of divergence-free functions. However, to test \eqref{eq:NS mod2} with a vector field having non-trivial divergence, we first need to associate a pressure $p_\rho$ with a solution $v_\rho$. We will obtain the pressure by a method that is well-established for the Navier-Stokes system, cf., e.g., \cite[Ch. 5]{RRS}. Let $(v_\rho, p_\rho)$ solve \eqref{eq:NS mod2}. We take the distributional divergence of the first equation in \eqref{eq:NS mod2} by testing it with $\nabla \varphi$ for $\varphi \in C^\i_c \left( (0, T) \times \R^3 \right)$ to obtain
	$$
	\DIV \DIV \left( \mathfrak{R} \circ v_\rho \right) = - \Delta p \quad \text{ in } \DD' \left( (0, T) \times \R^3 \right).
	$$
	Equivalently, $p$ solves the equation
	\begin{equation} \label{eq:press}
		- \Delta p = \sum_{i, j = 1}^3 \p_i \p_j \left( \mathfrak{r} (v_\rho)_i \mathfrak{r}(v_\rho)_j \right).
	\end{equation}
	From this, one can uniquely solve for $p$ by imposing integrability on $p$. Invoking the Calderón-Zygmund theory of singular integral operators just as done for \eqref{eq:NS}, additional estimates for $p$ follow. Here, the divergence structure of our regularized system \eqref{eq:NS mod} enters; because we have replaced $u$ by $\mathfrak{r} \circ v_\rho$ throughout the convective term, this transfer of structure from the Navier-Stokes system to the approximate one becomes possible.
	
	\begin{lemma} \label{lem:p sol}
		Let $z \in L^2(\R^3; \R^3)$ and suppose that $p \colon \R^3 \to \R$ is the unique function that is integrable to some power and satisfies
		\begin{equation} \label{eq:pressure}
			- \Delta p = \sum_{i, j = 1}^3 \p_i \p_j \left( z_i z_j \right).
		\end{equation}
		Then, for any $r \in \left( 1, \i \right)$, we have
		\begin{equation} \label{eq:1st CaldZyg}
			\left| p \right|_{L^r(\R^3) } \le C_r \left| z \right|^2_{L^{2r}(\R^3) } = C_r \left| \| z \|^2 \right|_{L^r(\R^3) }
		\end{equation}
		and
		\begin{equation} \label{eq:2nd CaldZyg}
			\left| \nabla p \right|_{L^r(\R^3) } \le C_r \left| \DIV \left( z \otimes z \right) \right|_{L^r(\R^3) }.
		\end{equation}
	\end{lemma}
	
	\begin{proof}
		The proof is analogous to \cite[Lem. 5.1]{RRS}. However, when adapting the argument to our setting, one must account for the possibility that $z$ has nontrivial divergence.
	\end{proof}
	
	\subsection{New energy identity}
	
	In this subsection, we obtain for the approximate system \eqref{eq:NS mod2} a tentative version of a new energy estimate by means of which our improved estimates for the Navier-Stokes system could follow.
	
	\begin{lemma} \label{lem:renorm id}
		Let $h \in C^2(\R^3)$ be a radially symmetric function satisfying $h(0) = 0$, $D h(0) = 0$, and $\| D^2 h \|_\i < \i$. For any initial value $u_0 \in H^1_\sigma(\R^3)$, let $v_\rho$ denote the unique solution to \eqref{eq:NS mod2} provided by \Cref{thm:weak appr sol e=zero}. Then $v_\rho$ satisfies the following energy identity:
		\begin{equation} \label{eq:renorm en id}
			\begin{gathered}
				\int_{\R^3} h(v_\rho(t) ) \, \mathrm{d} x
				+ \sum_{i = 1}^3 \int_0^t \int_{\R^3} D^2 h(v_\rho) \left[ \p_i v_\rho, \p_i v_\rho \right] \, \mathrm{d}x \, \mathrm{d} s \\
				= \int_{\R^3} h(u_0) \, \mathrm{d} x
				+ \int_0^t \int_{\R^3} p \DIV \left( D h \circ v_\rho \right) \, \mathrm{d} x \, \mathrm{d} s \\
				= \int_{\R^3} h(u_0) \, \mathrm{d} x
				+ \int_0^t \int_{\R^3} p \sum_{i = 1}^3 D \p_i h \left( v_\rho \right) \cdot \p_i v_\rho \, \mathrm{d} x \, \mathrm{d} s \quad \forall t \in \cl \left[ 0, T \right).
			\end{gathered}
		\end{equation}
		Moreover, if $h(u) = \tfrac{1}{r} \left| u \right|^r$ for $r \in \left[ 2, 5 \right)$, then \eqref{eq:renorm en id} remains valid provided that
		\begin{equation} \label{eq:assu in en fin}
			\int_{\R^3} h(u_0) \, \mathrm{d} x < \i.
		\end{equation}
	\end{lemma}
	
	\begin{proof}
		From the assumptions $h(0) = 0$, $D h(0) = 0$, and $\| D^2 h \|_\i < \i$, we obtain
		\begin{equation} \label{eq:h grwth ctrl}
			\left| D h(x) \right| \left| x \right| + \left| h(x) \right| \le C \left| x \right|^2 \quad \forall x \in \R^3.
		\end{equation}
		By the Krasnoselskii Theorem \cite[Ch. 4, Prop. 1.1]{ET} and the Sobolev chain rule \cite[Thm. 4.4]{EG}, the growth condition $\| D^2 h \|_\i < \i$ and \eqref{eq:h grwth ctrl} imply that $h$ induces a continuous superposition operator acting from $L^r$ to $L^{r / 2}$ and $D h$ induces such an operator acting from $L^r$ to $L^r$ and from $W^{1, r}$ to $W^{1, r}$ for $1 \le r \le \i$. Clearly, the former operator is the Frechet primitive of the latter if both act on the same Lebesgue space for $r \ge 2$. Therefore, the chain rule implies
		\begin{equation} \label{eq:chain rule appli}
			\left( h \circ u \right)' = D h (u) \cdot u' \text{ in } \DD'(0, T)
		\end{equation}
		for every $u \in H^1_{\mathrm{loc} } ( 0, T ; L^2(\R^3; \R^3) )$. This is clear if $u$ is smooth and in the general case follows by regularization as in the scalar-valued case, cf. \cite[Thm. 4.4]{EG}. Thus, recalling the first half of \eqref{eq:weak slt reg e=zero}, we get
		\begin{equation} \label{eq:1st}
			\int_0^t \int_{\R^3} \p_t v_\rho \cdot D h(v_\rho) \, \mathrm{d} x \, \mathrm{d} s
			= \int_{\R^3} h(v_\rho(t) ) \, \mathrm{d} x
			- \int_{\R^3} h(u_0) \, \mathrm{d} x
		\end{equation}
		for all $t \in \cl \left[ 0, T \right)$. Next, we claim that
		\begin{equation} \label{eq:2nd, 1st claim}
			\int_{\R^3} \langle \DIV \left( \mathfrak{R} \circ v_\rho \right)(t), D h \circ v_\rho(t) \rangle \, \mathrm{d} x = 0 \quad \text{ for a.e. } t \in (0, T).
		\end{equation}
		By \eqref{eq:weak slt reg e=zero}, there is a set $S \subset (0, T)$ of full measure such that $v_\rho(t)$ belongs to $H^1_\sigma(\R^3)$ for all $t \in S$. Hence, by continuity of the involved superposition operators and denseness of test functions in $H^1_\sigma(\R^3)$, it suffices to prove \eqref{eq:2nd, 1st claim} for all $v \in C^\i_{c, \sigma}(\R^3)$ instead of the vector $v_\rho(t)$ for every $t \in S$. Here, the composite mapping of the superposition operator with the divergence operator is continuous from $H^1(\R^3; \R^3)$ to $L^2(\R^3)$ by \eqref{eq:R superpos Sobolev}. Since $h$ is radially symmetric, there exists a function $\mathfrak{h} \in C^1(\R^+)$ with $\nabla h(x) = \mathfrak{h}( \left| x \right| ) x$ for $x \neq 0$. Integrating by parts, using $\sum_{j = 1}^3 \left| v_j \right|^2 = \left| v \right|^2$, and invoking \Cref{lem:convec van}, we obtain
		\begin{align*}
			& - \int_{\R^3} \langle \DIV \left( \mathfrak{R} \circ v \right), D h \circ v \rangle \, \mathrm{d} x
			= \int_{\R^3} \langle \mathfrak{R} \circ v , \nabla \left( D h \circ v \right) \rangle \, \mathrm{d} x \\
			& = \sum_{i, j = 1}^3 \int_{\R^3} \rho^2 \left( \left| v \right| \right) v_i v_j \p_i \left( \mathfrak{h}( \left| v \right| ) v_j \right) \, \mathrm{d} x \\
			& = \sum_{i, j = 1}^3 \int_{\R^3} \rho^2 \left( \left| v \right| \right) \mathfrak{h}( \left| v \right| ) v_i v_j \p_i v_j \, \mathrm{d} x
			+ \sum_{i, j = 1}^3 \int_{\R^3} \rho^2 \left( \left| v \right| \right) v_i \left| v_j \right|^2 \p_i \left( \mathfrak{h}( \left| v \right| ) \right) \, \mathrm{d} x \\
			& = \int_{\R^3} \rho^2 \left( \left| v \right| \right) \mathfrak{h}( \left| v \right| ) \langle \left( v \cdot \nabla \right) v, v \rangle \, \mathrm{d} x 
			+ \sum_{i, j = 1}^3 \int_{\R^3} \rho^2 \left( \left| v \right| \right) v_i \left| v_j \right|^2 \p_i \left( \mathfrak{h}( \left| v \right| ) \right) \, \mathrm{d} x \\
			& = \sum_{i, j, k = 1}^3 \int_{\R^3} \rho^2 \left( \left| v \right| \right) v_i \left| v_j \right|^2 \mathfrak{h}'( \left| v \right| ) \frac{\p_i v_k v_k}{\left| v \right|} \, \mathrm{d} x \\
			& = \sum_{i, k = 1}^3 \int_{\R^3} \rho^2 \left( \left| v \right| \right) \left| v \right| v_i \p_i v_k v_k \, \mathrm{d} x \\
			& = \int_{\R^3} \rho^2 \left( \left| v \right| \right) \left| v \right| \langle \left( v \cdot \nabla \right) v, v \rangle \, \mathrm{d} x \\
			& = 0
		\end{align*}
		whence \eqref{eq:2nd, 1st claim} follows. Integrating \eqref{eq:2nd, 1st claim} over $(0, t)$, we obtain
		\begin{equation} \label{eq:2nd}
			\int_0^t \int_{\R^3} \langle \DIV \left( \mathfrak{R} \circ v_\rho \right), D h (v_\rho) \rangle \, \mathrm{d} x \, \mathrm{d} s = 0.
		\end{equation}
		Since $D h \circ v_\rho \in L^\i_{\mathrm{loc} } \left( \cl \left[ 0, T \right) ; H^1(\R^3; \R^3) \right)$ by \eqref{eq:weak slt reg e=zero}, differentiation gives
		\begin{equation} \label{eq:3rd}
			\begin{gathered}
				\int_0^t \int_{\R^3} \sum_{i = 1}^3 \nabla v_{\rho, i} \cdot \nabla \left( D h \circ v_\rho \right)_i \, \mathrm{d} x \, \mathrm{d} s \\
				= \sum_{i = 1}^3 \int_0^t \int_{\R^3} D^2 h(v_\rho) \left[ \p_i v_\rho, \p_i v_\rho \right] \, \mathrm{d} x \, \mathrm{d} s.
			\end{gathered}
		\end{equation}
		Finally, to demonstrate that the fourth term
		\begin{equation} \label{eq:4th}
			\int_0^t \int_{\R^3} \langle \nabla p, D h (v_\rho) \rangle \, \mathrm{d} x \, \mathrm{d} s
		\end{equation}
		is well-defined, we observe that
		\begin{equation} \label{eq: Dh vrho ests}
			D h \circ v_\rho \in L^\i( 0, T ; L^2(\R^3) ), \quad \nabla \left( D h \circ v_\rho \right) \in L^2( (0, T) \times \R^3)
		\end{equation}
		by \eqref{eq:Ham appr en id2} and \eqref{eq:h grwth ctrl}. Also
		\begin{equation} \label{eq:p ests}
			\begin{gathered}
				\left| p \right|_{L^2((0, T) \times \R^3) } \le C \left| \mathfrak{R} \circ v_\rho \right|_{L^2( (0, T) \times \R^3) } < \i, \\
				\left| \nabla p \right|_{L^2((0, T) \times \R^3) } \le C \left| \DIV \left( \mathfrak{R} \circ v_\rho \right) \right|_{L^2((0, T) \times \R^3) } \le C \| D \mathfrak{R} \|_\i \left| \nabla v_\rho \right|_{L^2((0, T) \times \R^3) } < \i
			\end{gathered}
		\end{equation}
		by \eqref{eq:Ham appr en id2}, \eqref{eq:press}, \Cref{lem:p sol}, \eqref{eq:R grwth} and \eqref{eq:R superpos Lebesgue}. Combining \eqref{eq: Dh vrho ests}, \eqref{eq:p ests}, we may integrate \eqref{eq:4th} by parts to get
		\begin{equation} \label{eq:ibp}
			\begin{gathered}
				\int_0^t \int_{\R^3} \langle \nabla p, D h(v_\rho) \rangle \, \mathrm{d} x \, \mathrm{d} s
				= - \int_0^t \int_{\R^3} p \DIV \left( Dh \circ v_\rho \right) \, \mathrm{d} x \, \mathrm{d} s \\
				= - \int_0^t \int_{\R^3} p \sum_{i = 1}^3 D \p_i h \left( v_\rho \right) \cdot \p_i v_\rho \, \mathrm{d} x \, \mathrm{d} s.
			\end{gathered}
		\end{equation}
		Here, the radial symmetry of $h$ enters in the last step to arrive at the alternative form of $\DIV \left( Dh \circ v_\rho \right)$. Finally, testing the first equation in \eqref{eq:NS mod2} by $D h \left( v_\rho \right)$, applying \eqref{eq:1st}, \eqref{eq:2nd}, \eqref{eq:3rd}, \eqref{eq:ibp} gives \eqref{eq:renorm en id}.
		\\
		Let $\phi \in C^2(\R^3)$ be a convex, non-negative, radially symmetric function satisfying $\phi(0) = 0$ and $D \phi(0) = 0$. Let $f(x_1) = \phi(x_1, 0)$ for $x_1 \ge 0$. For $x \ne 0$, we have
		\begin{equation} \label{eq:phi derivs}
			\begin{gathered}
				D \phi(x)[v] = f' \left( \left| x \right| \right) \frac{\langle x, v \rangle}{\left| x \right|}, \\
				D^2 \phi(x)[u, v] = f'' \left( \left| x \right| \right) \frac{\langle x, u \rangle \langle x, v \rangle}{\left| x \right|^2} + f' \left( \left| x \right| \right) \frac{\left| x \right|^2 \langle u, v \rangle - \langle x, u \rangle \langle x, v \rangle}{\left| x \right|^3}.
			\end{gathered}
		\end{equation}
		This structure is favorable because it allows us to approximate $\phi$ and $D^2 \phi[v, v]$ monotonically if we manage to approximate $f$, $f'$, and $f''$ in such a way.
		We obtain such an approximation by setting
		$$
		f_n(0) = f'_n(0) = 0, \quad f''_n(r) = \min\{ n, f''(r) \}. 
		$$
		Setting $\phi_n(x) = f_n \left( \left| x \right| \right)$, we see from \eqref{eq:phi derivs} that the functions $\phi_n$ are admissible choices for $h$ in \eqref{eq:renorm en id}. By the monotone convergence theorem, we conclude
		\begin{equation} \label{eq:mon cvg1}
			\begin{gathered}
				\lim_{n \to \i} \int_{\R^3} \phi_n(v_\rho(t) ) \, \mathrm{d} x
				+ \sum_{i = 1}^3 \int_0^t \int_{\R^3} D^2 \phi_n \left( v_\rho \right) \left[ \p_i v_\rho, \p_i v_\rho \right] \, \mathrm{d}x \, \mathrm{d} s
				\\
				= \int_{\R^3} \phi(v_\rho(t) ) \, \mathrm{d} x
				+ \sum_{i = 1}^3 \int_0^t \int_{\R^3} D^2 \phi \left( v_\rho \right) \left[ \p_i v_\rho, \p_i v_\rho \right] \, \mathrm{d}x \, \mathrm{d} s
			\end{gathered}
		\end{equation}
		and
		\begin{equation} \label{eq:mon cvg2}
			\lim_{n \to \i} \int_{\R^3} \phi_n(u_0) \, \mathrm{d} x = \int_{\R^3} \phi(u_0) \, \mathrm{d} x.
		\end{equation}
		Neither \eqref{eq:mon cvg1} nor \eqref{eq:mon cvg2} implies that the limit integral is finite. For \eqref{eq:mon cvg2}, finiteness is ensured by imposing an appropriate integrability assumption on the initial value $u_0$. For \eqref{eq:mon cvg1}, however, the behavior also depends on the third limit
		\begin{equation} \label{eq:dom cvg}	
			\lim_{n \to \i} \int_0^t \int_{\R^3} p \sum_{i = 1}^3 D \p_i \phi_n \left( v_\rho \right) \cdot \p_i v_\rho \, \mathrm{d} x \, \mathrm{d} s.
		\end{equation}
		Even though the limit in \eqref{eq:dom cvg} exists whenever the limit in \eqref{eq:mon cvg2} is finite, we still need additional justification to compute it under the integral sign and guarantee its finiteness. Let now $\phi(x) = \tfrac{1}{r} \left| x \right|^r$ for $r \in \left[ 2, 5 \right)$. We have
		$$
		\mathfrak{r} \circ v_\rho \in L^\i( (0, T) \times \R^3) \cap L^\i(0, T; L^2(\R^3) )
		$$
		by \eqref{eq:r grwth} and \eqref{eq:Ham appr en id2}. Therefore
		\begin{equation} \label{eq:press grwth}
			p \in L^\i( 0, T ; L^\alpha(\R^3) ) \quad \forall \alpha \in \left[ 1, \i \right)
		\end{equation}
		by \eqref{eq:press} and \eqref{eq:1st CaldZyg}. Hence, using $\DIV v_\rho = 0$, we obtain
		\begin{equation} \label{eq:dom cvg2}
			\begin{gathered}
				- \int_0^t \int_{\R^3} \langle \nabla p, D \phi_n \circ v_\rho \rangle \, \mathrm{d} x \, \mathrm{d} s
				= \int_0^t \int_{\R^3} p \DIV \left( \frac{f_n'( \left| v_\rho \right| ) }{\left| v_\rho \right|} v_\rho \right) \, \mathrm{d} x \, \mathrm{d} s \\
				= \sum_{i = 1}^3 \int_0^t \int_{\R^3} p \frac{\left| v_\rho \right| f_n''( \left| v_\rho \right| ) - f_n'( \left| v_\rho \right| ) }{\left| v_\rho \right|^3} v_\rho \cdot \p_i v_\rho v_{\rho, i} \, \mathrm{d} x \, \mathrm{d} s.
			\end{gathered}
		\end{equation}
		To commute limit and integral sign, it suffices to establish a majorant for the factor
		$$
		\frac{\left| v_\rho \right| f_n''( \left| v_\rho \right| ) - f_n'( \left| v_\rho \right| ) }{\left| v_\rho \right|^3} v_\rho \cdot \p_i v_\rho v_{\rho, i}
		$$
		that will provide an integrable majorant to its product with the pressure. Due to
		$$
		0 \le f_n'(x) \le f'(x) = \left| x \right|^{r - 1}
		\text{ and }
		0 \le f_n''(x) \le f''(x) = (r - 1) \left| x \right|^{r - 2}
		$$
		such a majorant is given by the absolute value of
		\begin{equation} \label{eq:DIV nab phi v}
			\DIV \left( \left| v_\rho \right|^{r - 2} v_\rho \right)
			= (r - 2) \sum_{i = 1}^3 \left| v_\rho \right|^{r - 4} v_\rho \cdot \p_i v_\rho v_{\rho, i},
		\end{equation}
		which is in $L^\i(0, T' ; L^s(\R^3) )$ for every $T' \in \cl \left[ 0, T \right)$ for some $s \in \left( 1, 2 \right]$ for every $r \in \left[ 2, 5 \right)$ since
		$$
		v_\rho \in L^\i(0, T' ; L^2(\R^3) ) \cap L^\i(0, T' ; L^6(\R^3) ), \quad \nabla v_\rho \in L^\i(0, T' ; L^2(\R^3) )
		$$
		by \eqref{eq:weak slt reg e=zero} and the Sobolev inequality. Combining this with \eqref{eq:press grwth}, we conclude that
		\begin{equation} \label{eq:dom cvg3}
			\begin{aligned}
				\lim_{n \to \i} & \int_0^t \int_{\R^3} p \sum_{i = 1}^3 D \p_i \phi_n(v_\rho) \cdot \p_i v_\rho \, \mathrm{d} x \, \mathrm{d} s \\
				= & \int_0^t \int_{\R^3} p \sum_{i = 1}^3 D \p_i \phi(v_\rho) \cdot \p_i v_\rho \, \mathrm{d} x \, \mathrm{d} s \\
				= & \int_0^t \int_{\R^3} p \DIV \left( D \phi \circ v_\rho \right) \, \mathrm{d} x \, \mathrm{d} s \quad \forall r \in \left[2, 5 \right).
			\end{aligned}
		\end{equation}
		Combining \eqref{eq:mon cvg1}, \eqref{eq:mon cvg2}, and \eqref{eq:dom cvg3} yields \eqref{eq:renorm en id} for $h(u) = \tfrac{1}{r} \left| u \right|^r$ as \eqref{eq:assu in en fin} is satisfied due to
		\begin{equation*}
			u_0 \in H^1_\sigma(\R^3) \subset L^2_\sigma(\R^3) \cap L^6(\R^3) \subset L^r(\R^3) \quad \forall r \in \left[ 2, 5 \right)
		\end{equation*}
		by the Sobolev inequality and interpolation for Lebesgue spaces.
	\end{proof}
	
	\begin{conjecture} \label{lem:cent en est}
		Let $v_\rho$ be the unique solution to \eqref{eq:NS mod2} with initial datum $u_0 \in H^1_\sigma(\R^3)$ provided by \Cref{thm:weak appr sol e=zero}. Let $r = 2 + 2 / \sqrt{3}$ and $\phi(x) = \tfrac{1}{r} \left| x \right|^r$. Then, for every $\e > 0$, there is a constant $C_\e = C_\e > 0$ and an exponent $a \in (0, \i)$ such that
		\begin{equation} \label{eq:cent en est}
			\begin{gathered}
				\int_{\R^3} \phi(v_\rho(t) ) \, \mathrm{d} x
				+ (1 - \e) \sum_{i = 1}^3 \int_0^t \int_{\R^3} D^2 \phi(v_\rho) \left[ \p_i v_\rho, \p_i v_\rho \right] \, \mathrm{d}x \, \mathrm{d} s \\
				\le \int_{\R^3} \phi(u_0) \, \mathrm{d} x + C_\e t \left( 1 + \| \rho \|^2_\i \right) \| u_0 \|_H^{2a}
				\quad \forall t \in \cl \left[ 0, T \right).
			\end{gathered}
		\end{equation}
	\end{conjecture}
	
	\begin{proof}
		Using \eqref{eq:DIV nab phi v} and $\nabla \phi(x) = \left| x \right|^{r - 2} x$, we estimate
		\begin{equation} \label{eq:sht trm est}
			\begin{aligned}
				\left| \int_{\R^3} p \DIV \left( D \phi \circ v_\rho \right) \, \mathrm{d} x \right|
				& = (r - 2) \left| \sum_{i = 1}^3 \int_{\R^3} p \left| v_\rho \right|^{r - 4} v_\rho \cdot \p_i v_\rho v_{\rho, i} \, \mathrm{d} x \right| \\
				& \le (r - 2) \left| p \right|_{L^{\frac{3}{2} r} } \left| \left| v_\rho \right|^{r - 2} \nabla v_\rho \right|_{L^{\frac{3r}{3r - 2} } }
			\end{aligned}
		\end{equation}
		by the Hölder inequality. Let $\alpha = 3r / (3r - 2)$. Note that $\alpha \le 2$ since $3r /2 \ge 2$. Thus, $2 / \alpha$ is an admissible Hölder exponent. We further estimate the right-most term from the last step in \eqref{eq:sht trm est} using Hölder:
		\begin{equation} \label{eq:conv furth est}
			\begin{aligned}
				\left| \left| v_\rho \right|^{r - 2} \nabla v_\rho \right|_{L^\alpha}
				& \le \left| \left| v_\rho \right|^{(\frac{r}{2} - 1) \alpha} \right|^{\frac{1}{\alpha} }_{L^{(\frac{2}{\alpha})'} } \left| \left| v_\rho \right|^{(\frac{r}{2} - 1) \alpha} \left| \nabla v_\rho \right|^\alpha \right|^{\frac{1}{\alpha} }_{L^{\frac{2}{\alpha} } } \\
				& = \left| \left| v_\rho \right|^{(\frac{r}{2} - 1) \frac{2\alpha}{2 - \alpha}} \right|^{\frac{2 - \alpha}{2\alpha} }_{L^1} \left| \left| v_\rho \right|^{\frac{r}{2} - 1} \left| \nabla v_\rho \right| \right|^{1 / \alpha^2 }_{L^2} \\
				& = \left| \left| v_\rho \right|^2 \right|^{\frac{2 - \alpha}{2\alpha} }_{L^1} \left| \left| v_\rho \right|^{\frac{r}{2} - 1} \left| \nabla v_\rho \right| \right|^{1 / \alpha^2 }_{L^2}.
			\end{aligned}
		\end{equation}
		In the last step, we used
		$$
		\left. \left( \frac{r}{2} - 1 \right) \frac{2\alpha}{2 - \alpha} \right|_{r = 2 + 2 / \sqrt{3} } = \left. \left( \frac{r}{2} - 1 \right) \frac{6r}{3r - 4} \right|_{r = 2 + 2 / \sqrt{3} } = 2.
		$$
		Next, we estimate the pressure using \eqref{eq:press}, \eqref{eq:1st CaldZyg}, the definition of $\mathfrak{r}$, and the Sobolev inequality, obtaining
		\begin{equation} \label{eq:press furth est}
			\begin{aligned}
				\left| p \right|_{L^{\frac{3}{2} r} }
				\le C \left| \left| \mathfrak{r} \circ v_\rho \right|^2 \right|_{L^{\frac{3}{2} r} }
				\le C \| \rho \|_\i^2 \left| \left| v_\rho \right|^2 \right|_{L^{\frac{3}{2} r} }
				& = C \| \rho \|_\i^2 \left| \left| v_\rho \right|^{\frac{r}{2} } \right|_{L^6}^{\frac{4}{r} } \\
				& \le C \| \rho \|_\i^2 \left| \nabla \left| v_\rho \right|^{\frac{r}{2} } \right|_{L^2}^{\frac{4}{r} } \\
				& \le C \| \rho \|_\i^2 \left| \left| v_\rho \right|^{\frac{r}{2} - 1} \nabla v_\rho \right|_{L^2}^{\frac{4}{r} }.
			\end{aligned}
		\end{equation}
		Inserting \eqref{eq:conv furth est} and \eqref{eq:press furth est} into \eqref{eq:sht trm est} gives
		\begin{equation} \label{eq:press furth est2}
			\left| \int_{\R^3} p \DIV \left( D \phi \circ v_\rho \right) \, \mathrm{d} x \right|
			\le C \| \rho \|_\i^2 \left| \left| v_\rho \right|^2 \right|^{\frac{2 - \alpha}{2\alpha} }_{L^1} \left| \left| v_\rho \right|^{\frac{r}{2} - 1} \left| \nabla v_\rho \right| \right|^{1 / \alpha^2 + \frac{4}{r} }_{L^2}.
		\end{equation}
		Crucially, we have
		$$
		\left. \frac{1}{\alpha^2} + \frac{4}{r} \right|_{r = 2 + 2 / \sqrt{3} }
		= \left. \frac{(3r-2)^2 + 36r}{9r^2} \right|_{r = 2 + 2 / \sqrt{3} } \\
		\approx 1.89 < 2.
		$$
		Hence, defining $\beta$ and $\beta'$ by
		$$
		\frac{2}{\beta} = \frac{1}{\alpha^2} + \frac{4}{r}, \quad \frac{1}{\beta} + \frac{1}{\beta'} = 1,
		$$
		we apply Young's inequality with exponents $(\beta, \beta')$ and inserted $\e > 0$ to further estimate \eqref{eq:press furth est2}, obtaining
		\begin{equation} \label{eq:ult est}
			\begin{gathered}
				\left| \int_{\R^3} p \DIV \left( D \phi \circ v_\rho \right) \, \mathrm{d} x \right| \\
				\le \e \left| \left| v_\rho \right|^{\frac{r}{2} - 1} \nabla v_\rho \right|_{L^2}^2 + C^{\beta'} \e^{- \beta' / \beta } \| \rho \|^{2 \beta'}_\i \left| \left| v_\rho \right|^2 \right|_{L^1}^{\frac{6r}{3r - 4} \beta'} \\
				\le \e \left| \left| v_\rho \right|^{\frac{r}{2} - 1} \nabla v_\rho \right|_{L^2}^2
				+ C^{\beta'} \e^{- \beta' / \beta } \left( 1 + \| \rho \|^2_\i \right) \left| \left| u_0 \right|^2 \right|_{L^1}^{\frac{6r}{3r - 4} \beta'}.
			\end{gathered}
		\end{equation}
		In the last step, we used that $\| v_\rho(t) \|_H^2 \le \| u_0 \|_H^2$ for $t \ge 0$ by \eqref{eq:Ham appr en id2}. We have $r \in (3, 4)$ so that \Cref{lem:renorm id} implies \eqref{eq:renorm en id} with $\phi$ instead of $h$. In analogy to \eqref{eq:phi derivs}, we compute
		\begin{equation} \label{eq:D2 phi Hess}
			D^2 \phi(x)[u, u] = (r - 2) \left| x \right|^{r - 2} \left| \langle x, u \rangle \right|^2 + \left| x \right|^{r - 2} \left| u \right|^2 \ge \left| x \right|^{r - 2} \left| u \right|^2.
		\end{equation}
		In particular
		$$
		\left| \left| v_\rho \right|^{\frac{r}{2} - 1} \nabla v_\rho \right|_{L^2}^2 \le \sum_{i = 1}^3 \int_{\R^3} D^2 \phi(v_\rho)[\p_i v_\rho, \p_i v_\rho] \, \mathrm{d} x.
		$$
		Therefore, inserting \eqref{eq:ult est} into \eqref{eq:renorm en id} and setting $a = \frac{6r}{3r - 4} \beta' \in (0, \i)$, we arrive at
		\begin{align*}
			& \phantom{=} \int_{\R^3} \phi(v_\rho(t) ) \, \mathrm{d} x
			+ \sum_{i = 1}^3 \int_0^t \int_{\R^3} D^2 \phi(v_\rho) \left[ \p_i v_\rho, \p_i v_\rho \right] \, \mathrm{d}x \, \mathrm{d} s \\
			& = \int_{\R^3} \phi(u_0) \, \mathrm{d} x
			+ \int_0^t \int_{\R^3} p \DIV \left( \nabla \phi \circ v_\rho \right) \, \mathrm{d} x \, \mathrm{d} s \\
			& \le \int_{\R^3} \phi(u_0) \, \mathrm{d} x + \e \int_0^t \int_{\R^3} \left| v_\rho \right|^{r - 2}\left| \nabla v_\rho \right|^2 \, \mathrm{d} x \, \mathrm{d} s \\
			& + C^{\beta'} \e^{- \beta' / \beta } \left( 1 + \| \rho \|_\i^2 \right) \int_0^t \left\{ \int_{\R^3} \left| u_0 \right|^2 \, \mathrm{d} x \right\}^a \, \mathrm{d} s \\
			& \le \int_{\R^3} \phi(u_0) \, \mathrm{d} x + \e \sum_{i = 1}^3 \int_0^t \int_{\R^3} D^2 \phi(v_\rho) \left[ \p_i v_\rho, \p_i v_\rho \right] \, \mathrm{d} x \, \mathrm{d} s \\
			& + C^{\beta'} \e^{- \beta' / \beta } t \left( 1 + \| \rho \|_\i^2 \right) \left\{ \int_{\R^3} \left| u_0 \right|^2 \, \mathrm{d} x \right\}^a.
		\end{align*}
		Rearranging terms, we conclude \eqref{eq:cent en est}.
	\end{proof}
	
	\section{Passage to the limit}
	
	For passing to a limit $\rho \to 1$ in the system \eqref{eq:NS mod2}, we choose the parametric family
	$$
	\rho_1(t) = \frac{1}{\sqrt{1 + t^2} }, \quad \rho_\lambda(t) = \rho_1(\lambda t), \quad \lambda > 0.
	$$
	\begin{equation} \label{eq:deri rho lam}
		\rho'_\lambda(t) = - \lambda^2 t \rho_\lambda(t)^3
	\end{equation}
	so that $\rho'_\lambda(0) = 0$ and $\rho_\lambda$ satisfies \eqref{eq:rho grwth}. Moreover
	\begin{equation} \label{eq:rho lam bnd}
		\| \rho_\lambda \|_\i = \rho_\lambda(0) = 1 \quad \forall \lambda > 0
	\end{equation}
	and
	\begin{equation} \label{eq:rho lam cvg}
		\rho_\lambda(t) \uparrow 1 \text{ for every } t \in \R \text{ as } \lambda \downarrow 0.
	\end{equation}
	
	\begin{conjecture} \label{lem:strong sol global uni}
		Let $u_0 \in H^1_\sigma(\R^3)$ be a given initial value. Then there exists a subsequence $v_{\lambda_n}$, with $\lambda_n \to 0$, of the family of solutions $v_\lambda = v_{\rho_\lambda}$ to \eqref{eq:NS mod2} and a limit curve $u$ such that, with $r = 2 + 2/\sqrt{3}$, the following convergences and estimates hold:
		\begin{subequations}
			\begin{equation} \label{eq:cvg 1}
				v_{\lambda_n} \weakast u \text{ in } L^\i(0, T; H);
			\end{equation}
			
			\begin{equation} \label{eq:cvg 2}
				\nabla v_{\lambda_n} \weak \nabla u \text{ in } L^2(0, T; H);
			\end{equation}
			
			\begin{equation} \label{eq:cvg 3}
				v_{\lambda_n} \weak u \text{ in } L^2(0, T; L^6(\R^3) );
			\end{equation}
			
			\begin{equation} \label{eq:cvg 4}
				\p_t v_{\lambda_n} \weak \p_t u \text{ in } L^{\frac{4}{3} } \left( 0, T; H^1_\sigma(\R^3)^* \right)
			\end{equation}
			
			\begin{equation} \label{eq:cvg 5}
				v_{\lambda_n} \weakast u \text{ in } L^\i(0, T'; L^r(\R^3) ) \quad \forall T' \in \cl \left[ 0, T \right);
			\end{equation}
			
			\begin{equation} \label{eq:cvg 6}
				v_{\lambda_n} \weak u \text{ in } L^r(0, T'; L^{3r}(\R^3) ) \quad \forall T' \in \cl \left[ 0, T \right);
			\end{equation}
			
			\begin{equation} \label{eq:cvg 7}
				\nabla \left| v_{\lambda_n} \right|^{\frac{r}{2} } \weak \nabla \left| u \right|^{\frac{r}{2} } \text{ in } L^2(0, T'; L^2(\R^3) ) \quad \forall T' \in \cl \left[ 0, T \right);
			\end{equation}
			
			\begin{equation} \label{eq:cvg 8}
				\p_t v_{\lambda_n} \weak \p_t u \text{ in }	L^2 \left( 0, T'; H^1_\sigma(\R^3)^* \right)	\quad \forall T' \in \cl \left[ 0, T \right);
			\end{equation}
			
			\begin{equation} \label{eq:cvg 9}
				v_{\lambda_n} \to u \text{ in } L^2_{\mathrm{loc} }( \cl \left[ 0, T \right) \times \R^3 );
			\end{equation}
			
			\begin{equation} \label{eq:cvg 10}
				\cl \left[ 0, T \right) \ni t_n \to t_0 \implies v_{\lambda_n}(t_n) \weak u(t_0) \text{ in } L^2_\sigma(\R^3) \cap L^r(\R^3);
			\end{equation}
			
			\begin{equation} \label{eq:en est pap}
				\begin{gathered}
					\int_{\R^3} \left| u(t, x) \right|^r \, \mathrm{d} x + \int_0^t \int_{\R^3} \left| u(s, x) \right|^{r - 2} \left| \nabla u(s, x) \right|^2 \, \mathrm{d} x \, \mathrm{d} s
					\\
					\le C \int_{\R^3} \left| u(0, x) \right|^r \, \mathrm{d} x + C t \left( 1 + \left\{ \int_{\R^3} \left| u(0, x) \right|^2 \, \mathrm{d} x \right\}^a \right) \quad \forall t \ge 0
				\end{gathered}
			\end{equation}
			with $a \in (0, \i)$ from \eqref{eq:cent en est}.
		\end{subequations}
		Furthermore, $u$ is smooth in the interior,
		\begin{equation} \label{eq:str sol int smooth}
			u \in C^\i \left( ( 0, T' ] \times \R^3 ; \R^3 \right) \quad \forall T' \in \cl \left[ 0, T \right),
		\end{equation}
		and $u$ is a strong solution to the Navier-Stokes system \eqref{eq:NS} on $\cl \left[ 0, T \right)$. In particular, $u$ is the unique global Leray-Hopf solution attaining the initial value $u_0$.
	\end{conjecture}
	
	\begin{conjecture} \label{thm:main}
		Let $u \colon (0, T) \times \R^3 \to \R^3$ be a Leray-Hopf weak solution of \eqref{eq:NS} with initial value $u_0 \in L^2_\sigma(\R^3)$. Then $u$ has no branching after the initial time, is a strong solution on every $\cl \left[ \e, T \right)$ for all $\e > 0$, and
		\begin{equation} \label{eq:u int smooth}
			u \in C^\i \left( (0, T) \times \R^3 ; \R^3 \right).
		\end{equation}
		If moreover $u_0 \in L^2_\sigma(\R^3) \cap L^3(\R^3)$, then $u$ is globally unique. If in addition $u_0 \in H^1_\sigma(\R^3)$, then $u$ is a strong solution on $\cl \left[ 0, T \right)$. Finally, if the initial value $u_0$ satisfies
		\begin{equation} \label{eq:init smooth}
			u_0 \in \bigcap_{m \in \N} H^m(\R^3),
		\end{equation}
		e.g., if $u_0$ belongs to the Schwartz space $\mathcal{S}(\R^3)$ of smooth functions whose derivatives of all orders decay rapidly at infinity, then
		\begin{equation} \label{eq:u zero smooth}
			u \in C^\i \left( \cl \left[ 0, T \right) \times \R^3 ; \R^3 \right).
		\end{equation}
		If the pressure $p \colon (0, T) \times \R^3 \to \R$ is determined so that $p \in L^\i \left( 0, T ; L^1(\R^3) \right)$, then $p$ is unique whenever $u$ is and analogous smoothness statements to \eqref{eq:u int smooth} and \eqref{eq:u zero smooth} hold for $p$ under the same assumptions as for $u$.
	\end{conjecture}
	
	\appendix
	
	\section{Tightness in Lebesgue Spaces} \label{sec:appendix}
	
	In this appendix, we present a characterization of tightness in Lebesgue spaces. It is analogous to the equi-integrability criterion of de La Vallée Poussin, where a superlinear integrand prevents the concentration of mass on sets of arbitrarily small measure. Similarly, we employ a sufficiently heavy weight function to prevent the escape of mass to sets of infinite measure.
	\\
	Let $\left( \Omega, \AA, \mu \right)$ be a measure space, and let $\AA_f \subset \AA$ be the ring of sets with finite $\mu$-measure. A subset $\FF \subset L^p(\mu)$, $1 \le p \le \i$, is said to be \emph{tight} if
	$$
	\forall \e > 0 \, \exists A \in \AA_f \colon \sup_{u \in \FF} \| \chi_{\Omega \setminus A} u \|_{L^p} < \e.
	$$
	This property is also referred to as \emph{equi-vanishing at infinity} in the $L^p$ sense. A measurable function $t \colon \Omega \to \left[ 0, \i \right]$ is said to \emph{tighten} if there exists a sequence $A_n \in \AA_f$ with
	$$
	\lim_{n \to \i} \essinf_{\omega \in \Omega \setminus A_n} t(\omega) = \i.
	$$
	For example, if $t \colon \R^d \to \left[ 0, \i \right]$ is Borel measurable and satisfies
	$$
	\lim_{\left| x \right| \to \i} t(x) = \i,
	$$
	then $t$ tightens with respect to any given (finite or infinite) Radon measure on $\R^d$. A tightening function $t$ is said to tighten $\FF$ in $L^p(\mu)$ if
	$$
	\sup_{u \in \FF} \| t u \|_{L^p} < \i.
	$$
	
	\begin{lemma} \label{lem:tightness}
		A set $\FF \subset L^p(\mu)$, $1 \le p \le \i$, is tight iff there exists a tightening measurable function $t \colon \Omega \to \left[ 0, \i \right]$ such that
		$$
		\sup_{u \in \FF} \| t u \|_{L^p} < \i.
		$$
	\end{lemma}
	
	\begin{proof}
		$\implies$: Let $\FF$ be tight. By definition
		$$
		\forall n \in \N \, \exists A_n \in \AA_f \colon \sup_{u \in \FF} \| \chi_{\Omega \setminus A_n} u \|_{L^p} \le 2^{-n}.
		$$
		We set
		$$
		t_n(\omega) = n \chi_{\Omega \setminus A_n}(\omega), \quad t(\omega) = \sum_{n = 1}^\i t_n(\omega).
		$$
		Then
		$$
		\essinf_{\omega \in \Omega \setminus A_n} t(\omega) \ge \essinf_{\omega \in \Omega \setminus A_n} t_n(\omega) = n \to \i
		$$
		so $t$ tightens. Finally, we have
		\begin{align*}
			\sup_{u \in \FF} \| t u \|_{L^p}
			& = \sup_{u \in \FF} \| \sum_{n = 1}^\i t_n u \|_{L^p}
			\le \sup_{u \in \FF} \sum_{n = 1}^\i \| t_n u \|_{L^p}
			= \sup_{u \in \FF} \sum_{n = 1}^\i \| n \chi_{\Omega \setminus A_n} u \|_{L^p} \\
			& \le \sum_{n = 1}^\i n \sup_{u \in \FF} \| \chi_{\Omega \setminus A_n} u \|_{L^p}
			\le \sum_{n = 1}^\i n 2^{- n} < \i.
		\end{align*}
		$\impliedby$: Conversely, assume there exists a measurable function $t$ that tightens $\FF$ in $L^p(\mu)$. Then, by definition of tightening,
		$$
		\forall n \in \N \, \exists A_n \in \AA_f \colon \essinf_{\omega \in \Omega \setminus A_n} t(\omega) \ge n.
		$$
		Hence
		$$
		\sup_{u \in \FF} \| \chi_{\Omega \setminus A_n} u \|_{L^p} \le \frac{\sup_{u \in \FF} \| t u \|_{L^p} }{n} \to 0
		$$
		as $n \to \i$, which shows that $\FF$ is tight.
	\end{proof}

\end{document}